\documentclass[11pt]{amsart}
\usepackage{amssymb}
\usepackage{mathrsfs}
\usepackage{enumerate}
\makeatletter
\@namedef{subjclassname@2010}{%
  \textup{2010} Mathematics Subject Classification}
\makeatother
\newtheorem{lemma}{Lemma}
\newtheorem{theorem}{Theorem}[section]

\newtheorem{corollary}{Corollary}
\newtheorem{definition}{Definition}
\newtheorem{remark}{Remark}

\begin{document}

\title{SOME IDENTITIES AND INEQUALITIES FOR G-FUSION FRAMES}

\author[R. Zarghami Farfar]{Ramazan Zarghami Farfar}
\address{Dapartement of Geomatic and Mathematical\\Marand Faculty of Technical and Engineering\\University of Tabriz\\, Iran\\}
\email{zarghamir@gmail.com}

\author[V. Sadri]{Vahid Sadri}
\address{Department of Mathematics, Faculty of Tabriz  Branch\\ Technical and Vocational University (TUV), East Azarbaijan
, Iran}
\email{vahidsadri57@gmail.com}

\author[R. Ahmadi]{Reza Ahmadi}
\address{Institute of Fundamental Sciences\\University of Tabriz\\, Iran\\}
\email{rahmadi@tabrizu.ac.ir}

\begin{abstract}
G-fusion frames which are obtained from the combination of g-frames and fusion frames were recently introduced for Hilbert spaces. In this paper, we present a new identity for g-frames which was given by Najati in a special case. Also, by using the idea of this identity and their dual frames, some equlities and inequalities will be presented for g-fusion frames.
\end{abstract}

\subjclass[2010]{Primary 42C15; Secondary 46C99, 41A58}

\keywords{g-frame, dual g-frame, g-fusion frame, dual g-fusion frame.}

\maketitle

\section{Introduction}
Recent developments in frame theory and their applications are the result of some mathematician's efforts in this topic (see \cite{ds}, \cite{ga}, \cite{fe}, \cite{aklr}, \cite{ck} and \cite{ch}). By more than half a century, this theory has a interesting applications in different branches of science such as the filter bank theory, signal and image processing, wireless communications, atomic systems and Kadison-Singer problem. In 2005, Balan, Casazza and others  were able to find some useful identities in the frames and fusion frames by studying the properties of the Parseval frames \cite{balan}.
In \cite{sad}, a special kind of frames is introduced which are called g-fusion frames that are the combination of g-frames and fusion frames. We  present those identities in these frames.

\section{Preliminaries}
Throughout this paper, $H$ and $K$ are separable Hilbert spaces, $\pi_{V}$ is the orthogonal projection from $H$ onto a closed subspace $V\subset H$ and $\mathcal{B}(H,K)$ is the collection of all the bounded linear operators of $H$ into $K$. If $K=H$, then $\mathcal{B}(H,H)$ will be denoted by $\mathcal{B}(H)$. 
Also, $\lbrace H_j\rbrace_{j\in\Bbb J}$ is a sequence of Hilbert spaces  and $\Lambda_j\in\mathcal{B}(H,H_j)$ for each $j\in\Bbb J$ where $\Bbb J$ is a subset of $\Bbb Z$.
The following lemmas  from Operator Theory are useful. 
\begin{lemma}(\cite{ga})\label{l0}
Let $V\subseteq H$ be a closed subspace, and $T$ be a linear  bounded operator on $H$. Then
$$\pi_{V}T^*=\pi_{V}T^* \pi_{\overline{TV}}.$$
\end{lemma}
\begin{lemma}\label{l1}
Let $u\in\mathcal{B}(H)$ be adjoint and $v:=au^2+bu+c$ where $a,b,c\in\Bbb R$. 
\begin{enumerate}
\item[(I)] If $a>0$, then
$$\inf_{\Vert f\Vert=1}\langle vf, f\rangle\geq\frac{4ac-b^2}{4a}.$$
\item[(II)] If $a<0$, then
$$\sup_{\Vert f\Vert=1}\langle vf, f\rangle\leq\frac{4ac-b^2}{4a}.$$
\end{enumerate}
\end{lemma}
\begin{lemma}\label{l2}
If $u,v$ are operators on $H$ satisfying $u+v=id_{H}$, then $u-v=u^2-v^2$.
\end{lemma}
\begin{proof}
We can write
$$u-v=2u-id_{H}=u^2-(id_H-u)^2=u^2-v^2.$$
\end{proof}
 We define the space $\mathscr{H}_2:=(\sum_{j\in\Bbb J}\oplus H_j)_{\ell_2}$ by
\begin{eqnarray*}
\mathscr{H}_2=\big\lbrace \lbrace f_j\rbrace_{j\in\Bbb J} \ : \ f_j\in H_j , \ \sum_{j\in\Bbb J}\Vert f_j\Vert^2<\infty\big\rbrace,
\end{eqnarray*}
with the inner product defined by
$$\langle \lbrace f_j\rbrace, \lbrace g_j\rbrace\rangle=\sum_{j\in\Bbb J}\langle f_j, g_j\rangle.$$
It is clear that $\mathscr{H}_2$ is a Hilbert space with pointwise operations.
\begin{definition}
We call the sequence $\lbrace \Lambda_j\rbrace_{j\in\Bbb J}$ a g-frame for $H$ with respect to $\{H_j\}_{j\in\Bbb J}$ if there exist $0<A\leq B<\infty$ such that for each $f\in H$
\begin{eqnarray}\label{g}
A\Vert f\Vert^2\leq\sum_{j\in\Bbb J}\Vert \Lambda_j f\Vert^2\leq B\Vert f\Vert^2.
\end{eqnarray} 
\end{definition}
If $A=B=1$, we call $\lbrace \Lambda_j\rbrace_{j\in\Bbb J}$ a Parseval g-frame.
The synthesis and analysis operators in  g-frames are defined by
\begin{align*}
T_{\Lambda}&:\mathscr{H}_2\longrightarrow H \ \ \ \ \ \ \ \ \ , \ \ \ \ \ \ T_{\Lambda}^*:H\longrightarrow\mathscr{H}_2\\
T_{\Lambda}(\lbrace f_j\rbrace_{j\in\Bbb J})&=\sum_{j\in\Bbb J} \Lambda_{j}^{*}f_j \ \ \ , \ \ \ \ T_{\Lambda}^*(f)=\lbrace \Lambda_j f\rbrace_{j\in\Bbb J}.
\end{align*}
Therefore, the g-frame operator is defined by
$$Sf=TT^*f=\sum_{j\in\Bbb J}\Lambda^*_j \Lambda_j f.$$
The operator $S$ is a bounded, positive and invertible.
If $\tilde{\Lambda}_j:=\Lambda_j S^{-1}$, then $\{\tilde{\Lambda}_j\}_{j\in\Bbb J}$ is called a (canonical) dual g-frame of $\lbrace \Lambda_j\rbrace_{j\in\Bbb J}$ and we can write\begin{eqnarray}\label{g1}
f=\sum_{j\in\Bbb J}\tilde{\Lambda}^{\ast}_{j}\Lambda_j f=\sum_{j\in\Bbb J}\Lambda^{\ast}_j\tilde{\Lambda}_j f.
\end{eqnarray}
If $\{\Lambda_j\}_{j\in\Bbb J}$ is a g-frame for $H$ with bounds $A$ and $B$, respectively, then $\{\tilde{\Lambda}_j\}_{j\in\Bbb J}$ is a g-frame for $H$ too with bounds $B^{-1}$ and $A^{-1}$, respectively (for more details, see \cite{sun}).

\begin{definition}
Let $W=\lbrace W_j\rbrace_{j\in\Bbb J}$ be a family of closed subspaces of $H$, $\lbrace v_j\rbrace_{j\in\Bbb J}$ be a family of weights, i.e. $v_j>0$. We say $\Lambda:=(W_j, \Lambda_j, v_j)$ is a  g-fusion frame for $H$ if there exists $0<A\leq B<\infty$ such that for each $f\in H$
\begin{eqnarray}\label{gf}
A\Vert f\Vert^2\leq\sum_{j\in\Bbb J}v_j^2\Vert \Lambda_j \pi_{W_j}f\Vert^2\leq B\Vert f\Vert^2.
\end{eqnarray} 
\end{definition}
We call $\Lambda$ a Parseval g-fusion frame if $A=B=1$. When the right hand of (\ref{gf}) holds, $\Lambda$ is called a g-fusion Bessel sequence for $H$ with bound $B$. Throughout this paper, $\Lambda$ will be a triple $(W_j, \Lambda_j, v_j)$ with $j\in\Bbb J$.

The synthesis and analysis operators in the g-fusion frames are defined by (for more details, see \cite{sad})
\begin{align*}
T_{\Lambda}&:\mathscr{H}_2\longrightarrow H \ \ \ \ \ \ \ \ \ , \ \ \ \ \ \ T_{\Lambda}^*:H\longrightarrow\mathscr{H}_2\\
T_{\Lambda}(\lbrace f_j\rbrace_{j\in\Bbb J})&=\sum_{j\in\Bbb J}v_j \pi_{W_j}\Lambda_{j}^{*}f_j \ \ \ , \ \ \ \ T_{\Lambda}^*(f)=\lbrace v_j \Lambda_j \pi_{W_j}f\rbrace_{j\in\Bbb J}.
\end{align*}
Thus, the g-fusion frame operator is given by
$$S_{\Lambda}f=T_{\Lambda}T^*_{\Lambda}f=\sum_{j\in\Bbb J}v_j^2 \pi_{W_j}\Lambda^*_j \Lambda_j \pi_{W_j}f.$$
Therefore
$$A\  id_{H}\leq S_{\Lambda}\leq B\  id_H.$$
This means that $S_{\Lambda}$ is a bounded, positive and invertible operator (with adjoint inverse) and we have
$$B^{-1}id_{H}\leq S_{\Lambda}^{-1}\leq A^{-1}id_H.$$
 So, we have the folloeing reconstruction formula for any $f\in H$
\begin{equation}\label{gf1}
f=\sum_{j\in\Bbb J}v_j^2 \pi_{W_j}\Lambda^*_j \Lambda_j \pi_{W_j}S^{-1}_{\Lambda}f
=\sum_{j\in\Bbb J}v_j^2 S^{-1}_{\Lambda}\pi_{W_j}\Lambda^*_j \Lambda_j \pi_{W_j}f.
\end{equation} 
Let $\tilde{\Lambda}:=(S^{-1}_{\Lambda}W_j, \Lambda_j \pi_{W_j}S_{\Lambda}^{-1}, v_j)$. Then, $\tilde{\Lambda}$ is called the \textit{(canonical) dual g-fusion frame}  of $\Lambda$. Hence, for each $f\in H$, we get (see \cite{sad})
\begin{align}\label{gfd}
f=\sum_{j\in\Bbb J}v_j^2\pi_{W_j}\Lambda^*_j\tilde{\Lambda_j}\pi_{\tilde{W_j}}f=
\sum_{j\in\Bbb J}v_j^2\pi_{\tilde{W_j}}\tilde{\Lambda_j}^*\Lambda_j\pi_{W_j}f,
\end{align}
where $\tilde{W_j}:=S^{-1}_{\Lambda}W_j  \ , \ \tilde{\Lambda_j}:=\Lambda_j \pi_{W_j}S_{\Lambda}^{-1}.$
Thus, we obtain
\begin{align}\label{inverse}
\langle S^{-1}_{\Lambda}f, f\rangle=\sum_{j\in\Bbb J}v_j^2\Vert\tilde{\Lambda}_j\pi_{\tilde{W}_j}f\Vert^2.
\end{align}
\section{The Main Results}
 Let $\lbrace \Lambda_j\rbrace_{j\in\Bbb J}$ be a g-frame for $H$ with respect to $\{H_j\}_{j\in\Bbb J}$ with boinds $A, B$ and $\{\tilde{\Lambda}_j\}_{j\in\Bbb J}$ be a (canonical) dual g-frame of $\lbrace \Lambda_j\rbrace_{j\in\Bbb J}$. Suppose that $\Bbb I\subseteq\Bbb J$  and let
\begin{align*}
S_{\Bbb I}&:H\rightarrow H\\
S_{\Bbb I}f&:=\sum_{j\in\Bbb I} \Lambda^{\ast}_j\tilde{\Lambda}_j f.
\end{align*}
We have
\begin{align*}
\Vert S_{\Bbb I}f\Vert^2&=\big(\sup_{\Vert h\Vert=1}\vert\langle S_{\Bbb I}f,h\rangle\vert\big)^2\\
&=\sup_{\Vert h\Vert=1}\big(\sum_{j}\vert\langle\tilde{\Lambda}_j f, \Lambda_j h\rangle\vert\big)^2\\
&\leq\sum_{j}\Vert\tilde{\Lambda}_j f\Vert^2 .\sup_{\Vert h\Vert=1}\sum_{j}\Vert \Lambda_j h\Vert^2\\
&\leq BA^{-1}\Vert f\Vert^2.
\end{align*}
Thus, $S_{\Bbb I}\in\mathcal{B}(H)$ and is positive. From (\ref{g1}), we obtain that $S_{\Bbb I}+S_{\Bbb I^c}=id_{H}$.
\begin{theorem}\label{t1}
For $f\in H$, we have
\begin{eqnarray*}
\sum_{j\in\Bbb I}\langle \tilde{\Lambda}_j f,\Lambda_j f\rangle-\Vert S_{\Bbb I}f\Vert^2 =\sum_{j\in\Bbb I^c}\overline{\langle \tilde{\Lambda}_j f,\Lambda_j f\rangle}-\Vert S_{\Bbb I^c}f\Vert^2
\end{eqnarray*}
where $\Bbb I^c$ is the complement of $\Bbb I$.
\end{theorem}
\begin{proof}
For each $f\in H$, we have
\begin{align*}
\sum_{j\in\Bbb I}\langle \tilde{\Lambda}_j f,\Lambda_j f\rangle-\Vert \sum_{j\in\Bbb I} \Lambda^{\ast}_j \tilde{\Lambda}_j f\Vert^2&=\langle S_{\Bbb I}f,f\rangle-\Vert S_{\Bbb I}f\Vert^2\\
&=\langle S_{\Bbb I}f,f\rangle-\langle S^{\ast}_{\Bbb I} S_{\Bbb I}f,f\rangle\\
&=\langle(id_{H}-S_{\Bbb I})^{\ast}S_{\Bbb I}f,f\rangle\\
&=\langle S^{\ast}_{\Bbb I^c}(id_{H}-S_{\Bbb I^c})f,f\rangle\\
&=\langle S^{\ast}_{\Bbb I^c}f,f\rangle-\langle S^{\ast}_{\Bbb I^c}S_{\Bbb I^c}f,f\rangle\\
&=\langle f,S_{\Bbb I^c}f\rangle-\langle S_{\Bbb I^c}f,S_{\Bbb I^c}f\rangle\\
&=\sum_{j\in\Bbb I^c}\langle \Lambda_j f,\tilde{\Lambda}_j f\rangle-\Vert \sum_{j\in\Bbb I^c} \Lambda^{\ast}_j \tilde{\Lambda}_j f\Vert^2\\
&=\sum_{j\in\Bbb I^c}\overline{\langle \tilde{\Lambda}_j f,\Lambda_j f\rangle}-\Vert \sum_{j\in\Bbb I^c} \Lambda^{\ast}_j \tilde{\Lambda}_j f\Vert^2
\end{align*}
and the proof is completed.
\end{proof}
Now, if $\lbrace \Lambda_j\rbrace_{j\in\Bbb J}$ is a Parseval g-frame, then $\tilde{\Lambda_j}=\Lambda_j$ and we obtain the following famous formula:
\begin{align*}
\sum_{j\in\Bbb I}\Vert\Lambda_jf\Vert^2-\Vert S_{\Bbb I}f\Vert^2=\sum_{j\in\Bbb I^c}\Vert\Lambda_j f\Vert^2-\Vert S_{\Bbb I^c}f\Vert^2,
\end{align*}
where $S_{\Bbb I}f=\sum_{j\in\Bbb I} \Lambda^{\ast}_j\Lambda_j f$.

The same can be presented for g-fusion frames. Let $\Lambda$ be a g-fusion frame for $H$ with (canonical) dual g-fusion frame $\tilde{\Lambda}=(\tilde{W_j}, \tilde{\Lambda_j}, v_j)$ where $\tilde{W_j}=S^{-1}_{\Lambda}W_j$ and $\tilde{\Lambda_j}=\Lambda_j \pi_{W_j}S_{\Lambda}^{-1}$. For simplicity, we show the following operator with the same symbol $S_{\Bbb I}$ in which again $\Bbb I$ is a finite subset of $\Bbb J$:
\begin{align}
S_{\Bbb I}f=\sum_{j\in\Bbb I}v_j^2\pi_{W_j}\Lambda^*_j\tilde{\Lambda_j}\pi_{\tilde{W_j}}f , \ \ \ \ \ (\forall f\in H).
\end{align}
It is easy to check that, $S_{\Bbb I}\in\mathcal{B}(H)$ and positive. Again, we have 
$$S_{\Bbb I}+S_{\Bbb I^c}=id_{H}.$$
\begin{remark}
Let $\Lambda$ be a Parseval g-fusion frame for $H$. Since $\mathcal{B}(H)$ is a $C^*$-algebra and $S_{\Bbb I}$ is a positive, so
$r(S_{\Bbb I})=\Vert S_{\Bbb I}\Vert$, where $r$ is the spectral radius. Thus
$$\max_{\lambda\in\sigma(S_{\Bbb I})}\vert\lambda\vert=r(S_{\Bbb I})\leq 1$$
and we conclude that $\sigma(S_{\Bbb I})\in[0, 1]$.
\end{remark}
\begin{theorem}\label{tg1}
Let $f\in H$, then
\begin{eqnarray*}
\sum_{j\in\Bbb I}v_j^2\langle \tilde{\Lambda}_j \pi_{\tilde{W_j}}f,\Lambda_j \pi_{W_j}f\rangle-\Vert S_{\Bbb I}f\Vert^2 =\sum_{j\in\Bbb I^c}v_j^2\overline{\langle \tilde{\Lambda}_j \pi_{\tilde{W_j}}f,\Lambda_j \pi_{W_j}f\rangle}-\Vert S_{\Bbb I^c}f\Vert^2.
\end{eqnarray*}
\end{theorem}
\begin{proof}
The proof follows by a similar argument to the proof of Theorem \ref{t1}.
\end{proof}
\begin{corollary}\label{cor1}
Let $\Lambda$ be a Parseval g-fusion frame for $H$. Then
\begin{small}
\begin{align*}
\sum_{j\in\Bbb I}v_j^2\Vert \Lambda_j \pi_{W_j}f\Vert^2-\Vert \sum_{j\in\Bbb I}v_j^2 \pi_{W_j}\Lambda^*_j &\Lambda_j \pi_{W_j}f\Vert^2 =\\
&=\sum_{j\in\Bbb I^c}v_j^2\Vert \Lambda_j \pi_{W_j}f\Vert^2-\Vert \sum_{j\in\Bbb I^c}v_j^2 \pi_{W_j}\Lambda^*_j \Lambda_j \pi_{W_j}f\Vert^2.
\end{align*}
\end{small}
Moreover
\begin{eqnarray*}
\sum_{j\in\Bbb I}v_j^2\Vert \Lambda_j \pi_{W_j}f\Vert^2+\Vert \sum_{j\in\Bbb I^c}v_j^2 \pi_{W_j}\Lambda^*_j \Lambda_j \pi_{W_j}f\Vert^2\geq
\frac{3}{4}\Vert f\Vert^2.
\end{eqnarray*}
\end{corollary}
\begin{proof}
If $f\in H$, we obtain
\begin{align*}
\sum_{j\in\Bbb I}v_j^2\Vert \Lambda_j \pi_{W_j}f\Vert^2+\Vert S_{\Bbb I^c}f\Vert^2&=
\big\langle(S_{\Bbb I}+S^2_{\Bbb I^c})f, f\big\rangle\\
&=\big\langle(S_{\Bbb I}+id_H-2S_{\Bbb I}+S^2_{\Bbb I})f, f\big\rangle\\
&=\langle(id_H-S_{\Bbb I}+S_{\Bbb I}^2)f, f\rangle.
\end{align*}
Now, by Lemma \ref{l1} for $a=1$, $b=-1$ and $c=1$ the inequality holds.
\end{proof}
\begin{corollary}\label{cor2}
Let $\Lambda$ be a Parseval g-fusion frame for $H$. Then
$$0\leq S_{\Bbb I}-S_{\Bbb I}^2\leq\frac{1}{4}id_H$$
\end{corollary}
\begin{proof}
We have $S_{\Bbb I}S_{\Bbb I^c}=S_{\Bbb I^c}S_{\Bbb I}$. Then $0\leq S_{\Bbb I}S_{\Bbb I^c}=S_{\Bbb I}-S_{\Bbb I}^2$. Also, by Lemma \ref{l1}, we get
$$S_{\Bbb I}-S_{\Bbb I}^2\leq\frac{1}{4}id_H.$$
\end{proof}
\begin{theorem}\label{t3.3}
Let $\Lambda$ be a g-fusion frame with  the g-fusion frame operator $S_{\Lambda}$. If $\Bbb I\subseteq\Bbb J$  and $f\in H$, then
\begin{align*}
\sum_{j\in\Bbb I}v_j^2\Vert\Lambda_j\pi_{W_j}f\Vert^2+\Vert S_{\Lambda}^{-\frac{1}{2}}S_{\Bbb I^c}f\Vert^2=\sum_{j\in\Bbb I^c}v_j^2\Vert\Lambda_j\pi_{W_j}f\Vert^2+\Vert S_{\Lambda}^{-\frac{1}{2}}S_{\Bbb I}f\Vert^2.
\end{align*} 
\end{theorem}
\begin{proof}
Let $\Theta_j:=\Lambda_j\pi_{W_j}S_{\Lambda}^{-\frac{1}{2}}$ and $X_j:=S_{\Lambda}^{-\frac{1}{2}}W_j$. In \cite{sad} has been showed that $(X_j, \Theta_j, v_j)$ is a Parseval g-fusion frame for $H$. Then, by Corollary \ref{cor1}, we have
\begin{small}
\begin{align*}
\sum_{j\in\Bbb I}v_j^2\Vert \Theta_j \pi_{X_j}f\Vert^2-\Vert \sum_{j\in\Bbb I}v_j^2 \pi_{X_j}\Theta^*_j &\Theta_j \pi_{X_j}f\Vert^2 =\\
&=\sum_{j\in\Bbb I^c}v_j^2\Vert \Theta_j \pi_{X_j}f\Vert^2-\Vert \sum_{j\in\Bbb I^c}v_j^2 \pi_{X_j}\Theta^*_j \Theta_j \pi_{X_j}f\Vert^2.
\end{align*}
\end{small}
By replacing $S_{\Lambda}^{\frac{1}{2}}f$ instead of $f$ and this fact that
\begin{align*}
\sum_{j\in\Bbb I}v_j^2 \pi_{X_j}\Theta^*_j \Theta_j \pi_{X_j}f&=\sum_{j\in\Bbb I}v_j^2 (\Theta^*_j\pi_{X_j})^* \Theta_j \pi_{X_j}f\\
&=\sum_{j\in\Bbb I}v_j^2(\Lambda_j \pi_{W_j}S_{\Lambda}^{-\frac{1}{2}}\pi_{X_j})^*\Lambda_j \pi_{W_j}S_{\Lambda}^{-\frac{1}{2}}\pi_{X_j}f\\
&=\sum_{j\in\Bbb I}v_j^2 S_{\Lambda}^{-\frac{1}{2}}\pi_{W_j}\Lambda_j^*\Lambda_j \pi_{W_j}S_{\Lambda}^{-\frac{1}{2}}f\\
&=S_{\Lambda}^{-\frac{1}{2}}S_{\Bbb I}S_{\Lambda}^{-\frac{1}{2}}f,
\end{align*}
the proof follows.
\end{proof}
\begin{corollary}\label{cor3}
Let $\Lambda$ be a g-fusion frame with the g-fusion frame operator $S_{\Lambda}$. If $\Bbb I\subseteq\Bbb J$, then
$$0\leq S_{\Bbb I}-S_{\Bbb I}S^{-1}_{\Lambda}S_{\Bbb I}\leq\frac{1}{4}S_{\Lambda}.$$
\end{corollary}
\begin{proof}
In the proof of Theorem \ref{t3.3}, we showed that
$$\sum_{j\in\Bbb I}v_j^2 \pi_{X_j}\Theta^*_j \Theta_j \pi_{X_j}f=S_{\Lambda}^{-\frac{1}{2}}S_{\Bbb I}S_{\Lambda}^{-\frac{1}{2}}f.$$
By Corollary \ref{cor2} we get
$$0\leq\sum_{j\in\Bbb I}v_j^2 \pi_{X_j}\Theta^*_j \Theta_j \pi_{X_j}f-\big(\sum_{j\in\Bbb I}v_j^2 \pi_{X_j}\Theta^*_j \Theta_j \pi_{X_j}f\big)^2\leq\frac{1}{4}id_H.$$
Therefore,
$$0\leq S_{\Lambda}^{-\frac{1}{2}}(S_{\Bbb I}-S_{\Bbb I}S^{-1}_{\Lambda}S_{\Bbb I})S_{\Lambda}^{-\frac{1}{2}}\leq\frac{1}{4}id_H$$
and the proof is completed.
\end{proof}
\begin{corollary}
Suppose that $\Lambda$ is a g-fusion frame with the g-fusion frame operator $S_{\Lambda}$. If $\Bbb I\subseteq\Bbb J$  and $f\in H$, then
\begin{align*}
\sum_{j\in\Bbb I}v_j^2\Vert\Lambda_j\pi_{W_j}f\Vert^2+\Vert S_{\Lambda}^{-\frac{1}{2}}S_{\Bbb I^c}f\Vert^2\geq\frac{3}{4}\Vert S_{\Lambda}^{-1}\Vert^{-1} \Vert f\Vert^2.
\end{align*} 
\end{corollary}
\begin{proof}
By Theorem \ref{t3.3} and Corollary \ref{cor1}, we can write
\begin{small}
\begin{align*}
\sum_{j\in\Bbb I}v_j^2\Vert\Lambda_j\pi_{W_j}f\Vert^2+\Vert S_{\Lambda}^{-\frac{1}{2}}S_{\Bbb I^c}f\Vert^2&=\sum_{j\in\Bbb I}v_j^2\Vert \Theta_j \pi_{X_j}S^{\frac{1}{2}}_{\Lambda}f\Vert^2+\Vert \sum_{j\in\Bbb I^c}v_j^2 \pi_{X_j}\Theta^*_j \Theta_j \pi_{X_j}S^{\frac{1}{2}}_{\Lambda}f\Vert^2\\
&\geq\frac{3}{4}\Vert S^{\frac{1}{2}}_{\Lambda}f\Vert^2\\
&=\frac{3}{4}\langle S_{\Lambda}f, f\rangle\\
&\geq\frac{3}{4}\Vert S_{\Lambda}^{-1}\Vert^{-1} \Vert f\Vert^2.
\end{align*}
\end{small}
\end{proof}
\begin{theorem}
Let $\Lambda$ be a Parseval g-fusion frame for $H$ and $\Bbb I\subseteq\Bbb J$. Then
\begin{enumerate}
\item [(I)] $0\leq S_{\Bbb I}-S_{\Bbb I}^2\leq\dfrac{1}{4}id_H$.
\item [(II)] $\dfrac{1}{2}id_H\leq S_{\Bbb I}^2+S_{\Bbb I^c}^2\leq\dfrac{3}{2}id_H$.
\end{enumerate}
\end{theorem}
\begin{proof}
(I). Since $S_{\Bbb I}+S_{\Bbb I^c}=id_H$, then $S_{\Bbb I}S_{\Bbb I^c}+S^2_{\Bbb I^c}=S_{\Bbb I^c}$. Thus,
$$S_{\Bbb I}S_{\Bbb I^c}=S_{\Bbb I^c}-S^2_{\Bbb I^c}=S_{\Bbb I^c}(id_H-S_{\Bbb I^c})=S_{\Bbb I^c}S_{\Bbb I}.$$
But, $\Lambda$ is a Parseval, so $0\leq S_{\Bbb I}S_{\Bbb I^c}=S_{\Bbb I}-S_{\Bbb I}^2$. On the other hand, by Lemma \ref{l2}, we get
$$S_{\Bbb I}-S_{\Bbb I}^2\leq\frac{1}{4}id_H.$$
(II). We saw that $S_{\Bbb I}S_{\Bbb I^c}=S_{\Bbb I^c}S_{\Bbb I}$, then by Lemma \ref{l2},
$$S_{\Bbb I}^2+S_{\Bbb I^c}^2=id_H-2S_{\Bbb I}S_{\Bbb I^c}=2S_{\Bbb I}^2-2S_{\Bbb I}+id_H\geq\frac{1}{2}id_H.$$
On the other hand, we have again by Lemma \ref{l2} and $0\leq S_{\Bbb I}-S_{\Bbb I}^2$,
$$S_{\Bbb I}^2+S_{\Bbb I^c}^2\leq id_H+2S_{\Bbb I}-2S_{\Bbb I}^2\leq\frac{3}{2}id_H.$$
\end{proof}
\begin{corollary}
Let $\Lambda$ be a g-fusion frame with the g-fusion frame operator $S_{\Lambda}$. If $\Bbb I\subseteq\Bbb J$, then
$$\frac{1}{2}S_{\Lambda}\leq S_{\Bbb I}S^{-1}_{\Lambda}S_{\Bbb I}-S_{\Bbb I^c}S^{-1}_{\Lambda}S_{\Bbb I^c}\leq\frac{3}{2}S_{\Lambda}.$$
\end{corollary}
\begin{proof}
The proof is similar to the proof of Corollary \ref{cor3}.
\end{proof}
\begin{theorem}
Let $\Lambda$ be a g-fusion frame with the g-fusion frame operator $S_{\Lambda}$. If $\Bbb I\subseteq\Bbb J$, then for any $f\in H$
\begin{small}
\begin{align*}
\sum_{j\in\Bbb I}v_j^2\Vert\Lambda_j\pi_{W_j}f\Vert^2-\sum_{j\in\Bbb J}v_j^2\Vert\tilde{\Lambda}_j\pi_{\tilde{W}_j}M_{\Bbb I}f\Vert^2=\sum_{j\in\Bbb I^c}v_j^2\Vert\Lambda_j\pi_{W_j}f\Vert^2-\sum_{j\in\Bbb J}v_j^2\Vert\tilde{\Lambda}_j\pi_{\tilde{W}_j}M_{\Bbb I^c}f\Vert^2
\end{align*}
\end{small}
where
\begin{align*}
M_{\Bbb I}f=\sum_{j\in\Bbb I}v_j^2\pi_{W_j}\Lambda^*_j\Lambda_j\pi_{W_j}f.
\end{align*}
\end{theorem}
\begin{proof}
Via the definition of $S_{\Lambda}$, it is clear that $M_{\Bbb I}+M_{\Bbb I^c}=S_{\Lambda}$. Therefore, $S^{-1}_{\Lambda}M_{\Bbb I}+S^{-1}_{\Lambda}M_{\Bbb I^c}=id_H$. Hence, by Lemma \ref{l2}
$$S^{-1}_{\Lambda}M_{\Bbb I}-S^{-1}_{\Lambda}M_{\Bbb I^c}=(S^{-1}_{\Lambda}M_{\Bbb I})^2-(S^{-1}_{\Lambda}M_{\Bbb I^c})^2.$$
Thus, for each $f,g\in H$ we obtain
\begin{equation*}
\langle S^{-1}_{\Lambda}M_{\Bbb I}f, g\rangle-\langle S^{-1}_{\Lambda}M_{\Bbb I}S^{-1}_{\Lambda}M_{\Bbb I}f, g\rangle=\langle S^{-1}_{\Lambda}M_{\Bbb I^c}f, g\rangle-\langle S^{-1}_{\Lambda}M_{\Bbb I^c}S^{-1}_{\Lambda}M_{\Bbb I^c}f, g\rangle.
\end{equation*}
We choose $g$ to be $g=S_{\Lambda}f$ and we can get
\begin{equation*}
\langle M_{\Bbb I}f, f\rangle-\langle S^{-1}_{\Lambda}M_{\Bbb I}f, M_{\Bbb I}f\rangle=\langle M_{\Bbb I^c}f, f\rangle-\langle S^{-1}_{\Lambda}M_{\Bbb I^c}f, M_{\Bbb I^c}f\rangle
\end{equation*}
and by (\ref{inverse}),  the proof is completed.
\end{proof}

\end{document}